\documentclass[12pt,a4wde]{article}
\usepackage{amsmath, amsfonts, amssymb, amsthm, euscript, amscd, latexsym, mathtools}
 \usepackage[pdftex,colorlinks=true,
                       pdfstartview=FitV,
                       linkcolor=blue,
                       citecolor=blue,
                       urlcolor=blue,
           ]{hyperref}

\def\eee#1{ \begin{equation} #1 \end{equation} }
\def\aa#1{ \begin{align*} #1 \end{align*} }
 \def\aaa#1{ \begin{align} #1 \end{align} }
\def\mm#1{ \begin{multline*} #1 \end{multline*} }
\def\mmm#1{ \begin{multline} #1 \end{multline} }

\newtheorem{thm}{\sc Theorem}
\newtheorem{lem}{\sc Lemma}

\newcommand{\lb}{\label}
\newcommand{\rf}{\eqref}

\newcommand{\im}{\mathop{\rm Im}\nolimits}

\newcommand{\sss}{\scriptscriptstyle}

\newcommand{\gt}{\geqslant}
\newcommand{\lt}{\leqslant}
\newcommand{\sub}{\subset}

\newcommand{\al}{\alpha}
\newcommand{\gm}{\gamma}
\newcommand{\Gm}{\Gamma}

\newcommand{\bi}{\begin{itemize}}
\newcommand{\ei}{\end{itemize}}

\newcommand{\sg}{\sigma}
\newcommand{\Sg}{\Sigma}

\newcommand{\om}{\omega}
\newcommand{\mc}{\mathcal}
\newcommand{\Om}{\Omega}

\newcommand{\td}{\tilde}

\newcommand{\<}{{_{\sss E}\langle}}
\renewcommand{\>}{\rangle_{\sss E^*}}

\newcommand{\x}{\times}
\newcommand{\mto}{\mapsto}
\newcommand{\E}{\mathbb E}
\newcommand{\PP}{\mathbb P}
\newcommand{\W}{\mathbb W}

\newcommand{\Tr}{{\rm Tr\,}}

\newcommand{\spn}{{\rm span}}

\newcommand{\ovl}{\overline}

\DeclareMathOperator{\ind}{\mathbb I}

\newcommand{\fdot}{\,\cdot\,}
\def\Rnu{{\mathbb R}}

\def\Nnu{{\mathbb N}}

\def\ffi{\varphi}

\long\def\symbolfootnote[#1]#2{\begingroup%
\def\thefootnote{\fnsymbol{footnote}}\footnote[#1]{#2}\endgroup}

\include{srctex.sty}

\begin{document}

\author{Evelina Shamarova}
 \title{A version of H\"ormander's theorem \\ in 2-smooth Banach spaces}
\date{}
  \maketitle
  \vspace{-10mm}
\begin{center}
\small
 Centro de Matem\'atica da Universidade do Porto\\
 Rua do Campo Alegre 687, 4169-007 Porto, Portugal\\
 Email: evelinas@fc.up.pt
\end{center}


\begin{abstract}
We consider a stochastic evolution equation
in a 2-smooth Banach space with a densely and continuously embedded Hilbert subspace.
We prove that under H\"ormander's bracket condition, 
the image measure of the solution law 
under any finite-rank bounded linear operator
is absolutely continuous with respect to the Lebesgue measure.
To obtain this result, we apply methods of the Malliavin calculus.
\end{abstract}

  \vspace{1mm}

\section{Introduction}
Let $E$ be a 2-smooth Banach space (below, we recall the definition), and
$H\sub E$ be a Hilbert subspace. Further let
$H$ be dense in $E$, and 
the canonical embedding $H \xhookrightarrow {} E$ be continuous.
We consider the following stochastic evolution equation in $E$:
\aaa{
\label{SDE}
\begin{split}
& dX_t = (AX_t+\al(X_t))dt + \sg(X_t)dW_t,\\
& X_0 = x,
\end{split}
}
where $W_t$ is an $H$-cylindrical Brownian motion, 
$A$ is a generator of a strongly continuous semigroup on $E$,
$\al$ is a function $E\to E$, and $\sg$ maps $E$ to the space of 
$\gm$-radonifying operators $H\to E$ (see \cite{neer1}) denoted by $\gm(H,E)$.
Further let $\{e_i\}_{i=1}^\infty$ denote an orthonormal basis in $H$.
We prove that if $X_t$ is a solution to \eqref{SDE},
$F: E \to \Rnu^k$ is a bounded linear operator of rank $k$,
then,
under H\"ormander's bracket condition applied to the infinite system
of vectors $\sg_i(x)=\sg(x)e_i$, $i=1,2,\ldots $, and $\sg_0(x)= Ax + \al(x) + 
\sum_{i=1}^\infty \sg'_i(x) \sg_i(x)$, 
the law of $FX_t$ for any fixed $t$ is absolutely continuous
with respect to the Lebesgue measure on $\Rnu^k$. 
Since not every Banach space suits for consideration of equation \rf{SDE},
we work in 2-smooth Banach spaces, where we can employ the theory of stochastic integration 
and stochastic evolution equations 
\cite{brz, brz1, dett1, dett2}.
We mention that there exists a large class of UMD Banach spaces
where the latter theory was developed as well (see \cite{neer2,neer3}).
However, UMD Banach spaces do not seem suitable for proving our main
result. 

Regularity of transition probabilities
for solutions to infinite-dimensional SDEs 
under H\"ormander-type assumptions has been studied by many authors
(see, for example, \cite{Mattingly,hypoelipticity,Heirer,Ocone}).
Also, we would like to mention the work \cite{Belopolskaya}, 
where the authors prove the existence of
the logarithmic derivative (see \cite{ASF,Belopolskaya}) for the
transition probability of the solution to a Banach space valued SDE. 
However, all of the above articles, except \cite{Belopolskaya},
deal with Hilbert space-valued SDEs, and, to the author's knowledge,
a Banach space version of H\"ormander's theorem is 
obtained for the first time.

The paper is organized as follows: in Section \ref{existence} we 
sketch the proof of the existence and uniqueness of the mild solution to
\rf{SDE} by methods developed in \cite{Belopolskaya}.
In Section \ref{S2},
we obtain an SDE for the Malliavin derivative of the solution
to \rf{SDE}. 
The concept of the Malliavin derivative of a Banach space-valued
random variable was introduced, for example, in \cite{jan-maas}. 
In Section \ref{Diff}, 
we prove the Fr\'echet differentiability of the solution to \rf{SDE}
with respect to the initial data, and show that the Fr\'echet derivative
is the unique solution to an SDE in $\gm(H,E)$. The latter space
is also 2-smooth which allows us
to apply the results of Section \ref{existence} on the existence of solutions.
In Section \ref{inverse}, under some additional assumptions, we prove the existence of 
the right inverse operator to the derivative from Section \ref{Diff}.
Finally, in Section \ref{Hormander}, 
we show the non-degeneracy of the Malliavin covariance matrix 
of $FX_t$, and, by this, the existence of a density of the law of
$FX_t$ with respect to the Lebesgue measure.
In fact, we obtain
an infinite-dimensional analog of Nualart's proof \cite{Nualart}.

We remark that for an infinite dimensional SDE, the Fr\'echet derivative
of the solution with respect to the initial data, in general, exists
only in the mean-square sense, i.e. as  
a bounded linear operator $E\to L_2(\Om,E)$, 
although it would be desirable for our construction
to have it as bounded operator $E\to E$ a.s.
In the finite dimensional case, the latter fact holds due to Kolmogorov's continuity theorem.
Thus, one of the main difficulties of this work
was to find assumptions under which the infinite-dimensional
Fr\'echet derivative and its right inverse operator are a.s. 
bounded operators $E\to E$. 
To make our results valid for a larger class of operators $A$, such 
as Laplacian and other differential operators, we avoid the assumption on $A$ to generate
a group, as it was imposed in \cite{hypoelipticity}, 
although it would significantly simplify our arguments.

Finally, we remark, that the 
theory of differentiability of measures, developed in \cite{ASF},
offers an alternative, to the Malliavin calculus, approach to regularity of 
transition probabilities. This approach was undertaken
in \cite{Belopolskaya}. 
However, SDEs considered in \cite{Belopolskaya} are not stochastic evolution
equations, and therefore, the existence and smoothness of the density
does not follow from \cite{Belopolskaya} automatically. 
The present article
considers the ``traditional" Malliavin calculus approach to H\"ormander's theorem.

%
\section{Existence of the mild solution}
\label{existence}
Let $(\Om, \mc F, \PP)$ be a probability space, 
$W_t$ be an $H$-cylindrical Brownian motion, and
$\mc F_t$ be the filtration generated by $W_t$.
We consider stochastic evolution equation \rf{SDE}
in a 2-smooth Banach space $E$. We recall that
a Banach space $E$ is called 2-smooth 
if there exists a constant $C>0$ so that for all $x$ and $y$ from $E$,
\aa{
\|x+y\|^2 + \|x-y\|^2 \lt 2\|x\|^2 + C\|y\|^2.
}
We prove the existence of a mild solution to \rf{SDE} on 
the interval $[0,T]$,  $T>0$, i.e.
an $\mc F_t$-adapted stochastic process $X_t$ satisfying
\aaa{
\label{mild_SDE}
X_t = e^{tA}x + \int_0^t e^{(t-s)A}\al(X_s)\, ds
+ \int_0^t e^{(t-s)A}\sg(X_s)\, dW_s,
}
where 
$x\in E$, $e^{tA}$ is the semigroup generated by $A$. We assume that
$\al: E\to E$ and  $\sg: E \to \gm(H,E)$ satisfy
the following Lipschitz and linear growth conditions: 
\begin{enumerate}
\item[\textbf{A1}] \hspace{2mm}
$\|\al(x)-\al(y)\|_E + \|\sg(x)-\sg(y)\|_{\gm(H,E)} \lt \gm_1 \|x-y\|_E$,\\
for all $x, y \in E$, and for some constant $\gm_1$.
\item[\textbf{A2}] \hspace{2mm}
$\|\al(x)\|_E + \|\sg(x)\|_{\gm(H,E)}
\lt \gm_2(1 + \|x\|_E)$,\\
for all $x, y \in E$, and for some constant $\gm_2$.
\end{enumerate}
For any Banach space $G$,
in the space $\mc F_t$-adapted $G$-valued stochastic processes
we introduce the norm:
\aaa{
\label{s2_norm}
\|\xi\|_{S_2(G)}^2 = \sup_{t\in [0,T]}\E\|\xi(t,\fdot)\|_G^2.
}
\begin{thm}
\label{thm0}
Let A1 and A2 hold. Then equation \eqref{mild_SDE}
has a unique solution in the space $S_2(E)$. This solution
has a continuous path modification.
\end{thm}
\begin{proof}
The scheme of the proof is similar to which was used in \cite{Belopolskaya}.
We will search for the solution in $S_2(E)$. 
Consider the map $\Gm : S_2(E) \to S_2(E)$,
\aaa{
\lb{Gam}
\Gm(X_t) = e^{tA}x + \int_0^t e^{(t-s)A}\al(X_s)\, ds
+ \int_0^t e^{(t-s)A}\sg(X_s)\, dW_s.
}
By the results of \cite{brz}, $\int_0^t e^{(t-s)A}\sg(X_s)\, dW_s$
is $E$-valued, and
\aa{
\Big\|\int_0^t e^{(t-s)A}\sg(X_s)\, dW_s\Big\|_E^2
\lt C \int_0^t \|e^{(t-s)A}\sg(X_s)\|^2_{\gm(H,E)} ds.
}
By A2, the map $\Gm$ is well-defined.
Then, A1 and usual stochastic integral estimates
imply that there exists a constant $K>0$ so that
for each pair $X$ and $X'$ from $S_2(E)$
\aa{
\sup_{t\in [0,T]}\E\|\Gm^n(X_t) - \Gm^n(X'_t)\|^2_E \lt \frac{K^n T^n}{n!}
\sup_{t\in [0,T]}\E\|X_t-X'_t\|_E^2.
}
Pick up the integer $n$ so that 
$\frac{K^n T^n}{n!}< 1$. Then $\Gm^n :S_2 \to S_2$ is a contraction map.
The unique
fixed point of the map $\Gm^n$ is also the unique fixed point of $\Gm$.
By the results of \cite{neer4}, the stochastic convolution in \rf{mild_SDE}
has a continuous version. This and Assumption A1 imply that
the solution $X_t$ also has a continuous version.
\end{proof}

\section{The Malliavin derivative of the solution}
\label{S2}
The Malliavin derivative of a Banach space-valued random variable 
was defined in \cite{jan-maas}, pp. 154-155.
Let $\mc H = L_2([0,T], H)$ be the Hilbert space 
where we consider the isonormal Gaussian process
$\W(h) = \int_0^t h(s)dW_s$, $h\in \mc H$.  
According to \cite{jan-maas},
the domain $\mathbb D^{1,2}$ of the Malliavin 
derivative operator $D$ is defined by the squared norm
\aa{
\|\xi\|_{\mathbb D^{1,2}}^2 = \|\xi\|_{L_2(\Om,E)}^2 + \|D\xi\|_{L_2(\Om,\gm(\mc H,E))}^2.
}
In the following, we will need the two lemmas below. 
Lemma \ref{le1} is proved in \cite{pronk-veraar} (Lemma 3.7).
\begin{lem}
\lb{le1}
Let $G$ be a reflexive Banach space.
Suppose $\xi_n\to \xi$ in $L_2(\Om,G)$ and there is a constant $C>0$ such that
\aaa{
\lb{asb}
\sup_n \E\|D\xi_n\|^2_{\gm(\mc H,G)} < C.
}
Then, $\xi\in \mathbb D^{1,2}$, $\E\|D\xi\|^2_{\gm(\mc H,G)} < C$, and, moreover,
there exists a weakly convergent subsequence $D\xi_{n_k}\to D\xi$.
\end{lem}
Notice that a 2-smooth Banach space is 
uniformly smooth, and, therefore, reflexive.
Hence, Lemma \ref{le1} holds with $G=E$.
Lemma \ref{chain-lem} below is a simple version of the chain rule. 
\begin{lem}
\lb{chain-lem}
Let $\xi\in \mathbb D^{1,2}$ and let $F:E\to E$ have a bounded continuous 
Fr\'echet derivative.
Then, $F(\xi)\in\mathbb D^{1,2}$, and
\aaa{
\lb{chain}
DF(\xi) = F'(\xi) D\xi.
}
\end{lem}
\begin{proof}
Take a sequence of smooth random variables $\xi_n$ that converges
to $\xi$ in $\mathbb D^{1,2}$. Clearly, $DF(\xi_n) = F'(\xi_n) D\xi_n$.
By boundedness of the derivative $F'$, $F(\xi_n)\to F(\xi)$ in $L_2(\Om,E)$.
Moreover, $DF(\xi_n)$ satisfy Assumption \rf{asb} by the fact of convergence 
$D\xi_n \to D\xi$
in $L_2(\Om,\gm(\mc H,E))$, by the boundedness of $\|F'(\xi_n)\|_{\mc L(E)}$ 
uniformly in $n$, and by the ideal property of $\gm$-radonifying operators.
By  Lemma \ref{le1}, $F(\xi)\in\mathbb D^{1,2}$.
Moreover, there is a subsequence $\xi_{n_k}$ so that 
$DF(\xi_{n_k})\to DF(\xi)$ weakly in $L_2(\Om,\gm(\mc H,E))$.
On the other hand,
\mmm{
\lb{oo1}
\|F'(\xi_{n_k})D\xi_{n_k} - F'(\xi)D\xi\|_{L_2(\Om,\gm(\mc H,E))}
\lt
\|F'(\xi_{n_k})(D\xi_{n_k}-D\xi)\|_{L_2(\Om,\gm(\mc H,E))} \\
+ \|\big(F'(\xi_{n_k}) - F'(\xi)\big)D\xi\|_{L_2(\Om,\gm(\mc H,E))}.
}
The first term on the right-hand side converges to zero
by the boundedness of $F'$ and by the ideal property of
$\gm(\mc H,E)$. As for the second term,
we can find a further subsequence of $\xi_{n_k}$, 
we denote it again by $\xi_{n_k}$, that converges to $\xi$ a.s.
Then, the second term in \rf{oo1} converges to zero by the boundedness
of $F'$, by Lebesgue's theorem, and, again, by the ideal property of $\gm(\mc H,E)$. 
This proves \rf{chain}.
\end{proof}
We will need the next assumption.
\begin{enumerate}
\item[\textbf{A3}]
The functions $\al: E\to E$, $\sg_i: E\to E$, $i=1,\ldots, n$, 
have bounded Fr\'echet derivatives.
Moreover, $\sg': E \to \gm(H,\mc L(E))$ is continuous. 
\end{enumerate}
\begin{thm}
\lb{thhh2} 
Suppose A3 is fulfilled. Then, $X_t\in \mathbb D^{1,2}$ for all $t\in [0,T]$. 
Moreover, $DX_t\in L_2(\Om\x [0,T],\gm(H,E))$, and
for $r\lt t$,  $D_rX_t$ satisfies the following equation in $\gm(H,E)$:
\mmm{
\label{SDE-MD}
D_r X_t  = e^{(t-r)A}\sg(X_r)  + \int_r^t e^{(t-s)A}\al'(X_s)\, D_rX_s
\, ds\\
+ \int_r^t e^{(t-s)A}\sg'(X_s)\, D_rX_s  \, dW_s.
}
For $r>t$, $D_rX_t = 0$. 
\end{thm}
\begin{proof}
First we note that $\gm(H,E)$ is a 2-smooth Banach space, and, therefore,
\rf{SDE-MD} is well-defined. 
We construct iterations by setting $X^{(0)}_t = x$,
and $X^{(n+1)}_t = \Gamma(X^{(n)}_t)$, where $\Gm$ is defined by \rf{Gam}. 
Notice that each successive iteration $X^{(n)}_t$ has a continuous version,
since, by the results of \cite{neer4}, the stochastic
convolution process has a continuous version.
We are going to prove by induction on $n$ that all successive iterations 
$X^{(n)}_t$ are in the domain $\mathbb D^{1,2}$. 
Clearly, $X^{(0)}_t\in \mathbb D^{1,2}$, and $D X^{(0)}_t=0$.
As the induction hypothesis, we assume the following:
1) $X^{(n)}_t \in \mathbb D^{1,2}$, 2) $DX^{(n)}_t\in L_2(\Om\x [0,T],\gm(H,E))$,
3) for each fixed $r>0$
the path of $D_rX^{(n)}_t$ is uniformly continuous on $[r,T]$ in the mean-square sense,
4) $D_rX^{(n)}_t = 0$ for $r>t$, 5) $\E\|D_r X^{(n)}_t\|^4_{\gm(H,E)}$ is bounded. 
Note that, 
by the induction hypothesis, we can evaluate $DX^{(n)}_t$ at any point $r\in [0,T]$,
and write $D_rX^{(n)}_t$ for this evaluation.
Let us prove these statements for $n+1$.
We start by showing the relation:
\aaa{
\lb{epp1}
D_rX^{(n+1)}_t = e^{(t-r)A}\sg(X^{(n)}_r) + \int_r^t e^{(t-s)A}\al'(X^{(n)}_s)D_rX^{(n)}_s\, ds\notag\\
+ \int_r^t e^{(t-s)A}\sg'(X^{(n)}_s)D_rX^{(n)}_s \, dW_s.
}
For this, we need to prove that
\aaa{
\lb{MD-sto}
&D_r \int_0^t e^{(t-s)A}\sg(X^{(n)}_s)\, dW_s = 
e^{(t-r)A}\sg(X^{(n)}_r) +
\int_r^t e^{(t-s)A}\sg'(X^{(n)}_s)D_rX^{(n)}_s \, dW_s\\
&\text{and}\notag\\
\lb{MD-riem}
&D_r  \int_0^t e^{(t-s)A}\al(X^{(n)}_s) \, ds
=
\int_r^t e^{(t-s)A}\al'(X^{(n)}_s)\, D_rX^{(n)}_s \, ds.
}
Note that the stochastic integral on the right-hand side of \rf{MD-sto} is well-defined.
Indeed, since $D_rX^{(n)}_s$ takes values in $\gm(H,E)$,
then, by A3, the integrand of the stochastic integral takes values 
in $\gm(H,\gm(H,E))$.
This implies (see \cite{neer1},\cite{brz}) that the stochastic integral in \rf{MD-sto}
is in $L_2(\Om,\gm(H,E))$, and, moreover, that there exists a constant $C>0$ so that 
\mm{
\Big\|\int_r^t e^{(t-s)A}\sg'(X^{(n)}_s)D_rX^{(n)}_s \, dW_s\Big\|_{\gm(H,E)}^2
\\ \lt C \int_r^t\|e^{(t-s)A}\sg'(X^{(n)}_s)D_rX^{(n)}_s\|^2_{\gm(H,\gm(H,E))}\, ds.
}
To prove \rf{MD-sto} and \rf {MD-riem}, suppose first that $r>t$.
Fix a partition $\mc P= \{0=t_0<t_1<\dots <t_N = t\}$ and
consider a simple integrand of the form
\aaa{
\lb{simple-int}
\sg_N(X^{(n)},s) = \sum_{i=1}^N  e^{(t-t_{i})A}\sg(X^{(n)}_{t_i})\ind_{(t_{i-1},t_i]}(s).
}
Note that $\sg_N(X^{(n)},s)$ converges to $\sg(X^{(n)}_s)$ in the mean-square sense
which is implied by the uniform continuity of paths of $X^{(n)}_s$ in the $L_2(\Om,E)$-norm.
The latter uniform continuity is implied by the relation $X^{(n)} = \Gm(X^{(n-1)})$,
where $\Gm$ is defined by \rf{Gam}, and by the fact that $\E\|X^{(n)}_t\|^2_E$
is bounded uniformly in $n$ and $t\in [0,T]$ which follows from the same relation 
and the usual stochastic integral estimates.
Then, from Lemma \ref{chain-lem} and from the equality $D_r (W_te_i) = e_i\ind_{[0,t]}(r)$, 
it follows that $D_r\int_0^t \sg_N(X^{(n)},s)dW_s = 0$ if $D_rX^{(n)}_t=0$.
By taking the limit as the mesh of $\mc P$ goes to 0, 
we obtain that $D_r \int_0^t e^{(t-s)A}\sg(X^{(n)}_s)dW_s = 0$.
Analogously, $D_r  \int_0^t e^{(t-s)A}\al(X^{(n)}_s) ds =0$ if $D_rX^{(n)}_t=0$.
This proves that  for $r>t$, $D_rX^{(n+1)}_t=0$.
Now take an $r\lt t$ and fix a partition $\mc P= \{0=t_0<t_1<\dots <t_N = t\}$ containing $r$.
We have:
\aaa{
\lb{conv-study1}
D_r \int_0^t \sg_N(X^{(n)},s)\, dW_s = 
e^{(t-r)A}\sg(X^{(n)}_r) +
\int_r^t D_r\sg_N(X^{(n)},s) \, dW_s,
}
where $D_r\sg_N(X^{(n)},s)$ is computed using \rf{simple-int}. 
The right-hand side of the above relation, considered as a function of $\om$ and $r$, 
converges to the right-hand side of \rf{MD-sto} in $L_2(\Om,\gm(H,E))$ 
pointwise in $r\in [0,t]$.
Indeed, there exists a constant $\gm>0$ so that
\mm{
\E\Big\|\int_r^t e^{(t-s)A}\sg'(X^{(n)}_s)D_rX^{(n)}_s \, dW_s -\int_r^t D_r\sg_N(X^{(n)},s) \, dW_s\Big\|^2\\
\lt \gm\Big[\Big(\sum_{i=1}^N\int_{t_{i-1}}^{t_i} \hspace{-2mm}\E\|e^{(t-s)A}\sg'(X^{(n)}_s) - e^{(t-t_i)A}\sg'(X^{(n)}_{t_i})\|^4 ds\Big)^\frac12
\hspace{-1mm}\Big(\int_r^t \hspace{-1mm}\E\|D_rX^{(n)}_s\|^4 ds\Big)^\frac12\\
+ \sum_{i=1}^N\int_{t_{i-1}}^{t_i}\E \|D_rX^{(n)}_s- D_rX^{(n)}_{t_i}\|^2\, ds\Big].
}
The right-hand side of the above inequality converges to 
zero by the uniform continuity of paths of $X^{(n)}_s$, Lebesgue's
theorem, and the induction hypothesis.
The convergence of the right-hand side of \rf{conv-study1} 
to the right-hand side of \rf{MD-sto}
holds also in $L_2(\Om\x [0,T],\gm(H,E))$, and,
therefore, in $L_2(\Om,\gm(\mc H,E))$ by the canonical embedding of
$L_2([0,T],\gm(H,E))$ into $\gm(\mc H,E)$ for type 2 Banach spaces (see \cite{neer1}).
Thus, equality
\rf{MD-sto} will be implied by It\^o's isometry, by the continuity of paths,
and by the closedness of the Malliavin derivative operator.
Equality \rf{MD-riem} follows from similar arguments.
Therefore, $X^{(n+1)}_t \in \mathbb D^{1,2}$,
$DX^{(n+1)}_t\in L_2(\Om\x [0,T],\gm(H,E))$, and 
relation \rf{epp1} holds.
This relation implies that the paths of $D_rX^{(n+1)}_t$
are continuous in the mean-square sense on $[r,T]$.
The same relation and the maximal inequality for
stochastic convolutions, proved in \cite{neer4}, imply
that $\E\|D_rX^{(n+1)}_t\|^4$ is bounded.
This completes the induction argument.

%
Now we would like to prove \rf{asb} for $\xi_n = X^{(n)}_t$.
Relation \rf{epp1} implies the estimate:
\aa{
\E\|D_rX^{(n+1)}_t\|^2_{\gm(H,E)} \lt K\big(1+ \int_r^t \E\|D_rX^{(n)}_s\|^2_{\gm(H,E)}\big),
}
where $K>0$ is a constant which does not depend on $r$. 
This implies that for all $n$
\aaa{
\lb{e1}
\E\|D_rX^{(n)}_t\|^2_{\gm(H,E)} \lt Ke^{KT}.
}
Integrating \rf{e1} from $0$ to $T$ and using the fact of the canonical
embedding of $L_2([0,T],\gm(H,E))$ into $\gm(\mc H,E)$ 
we obtain that $D_rX^{n}_t$ takes values in $\gm(\mc H,E)$, and
\aaa{
\lb{e2}
\E\|DX^{(n)}_t\|^2_{\gm(\mc H,E)} \lt J\,K\,\E\int_0^T  \E\|D_rX^{(n)}_t\|^2_{\gm(H,E)}\, dt
\lt J\, K\, T\, e^{KT}
}
where $J>0$ is the embedding constant.  
By the results of Section \ref{existence}, $X^{(n)}_t\to X_t$ in $L_2(\Om,E)$.
Hence, by Lemma \ref{le1}, $X_t\in \mathbb D^{1,2}$,
and, moreover, there is a weakly convergent subsequence $DX^{(n_k)}_t\to DX_t$.
By \rf{e2}, this subsequence contains a further subsequence which converges
in $L_2(\Om \x [0,T],\gm(H,E)))$, say, to an element $\zeta$. Then, again
by the canonical embedding of $L_2([0,T],\gm(H,E))$ into $\gm(\mc H,E)$,
$\zeta=DX_t$. The latter implies that that we can evaluate $DX_t$ at $r\in [0,T]$,
and, moreover, $D_rX_t$ takes values in $\gm(H,E)$.

It remains to show \rf{SDE-MD}.
Take a $y'\in E^*$, and apply the functional $y'$, and then the operator $D_r$, to the both parts of \rf{mild_SDE}.
Using the result from \cite{Carmona} on the Malliavin derivative of the stochastic
integral (Proposition 5.4), for $r\lt t$ we obtain:
\aa{
\<D_rX_t,y'\> = \<e^{(t-r)A}\sg(X_r),y'\>  + \int_r^t \<e^{(t-s)A}\al'(X_s)\, D_rX_s,y'\>
\, ds\\
+ \int_r^t \<e^{(t-s)A}\sg'(X_s)\, D_rX_s  \, dW_s,y'\>.
}
This implies \rf{SDE-MD} since the above equation holds for all $y'\in E^*$, and  
\rf{SDE-MD} is a well-defined $\gm(H,E)$-valued stochastic evolution equation. 
If $r>t$, then $D_rX_t=0$ since it is the limit of  $DX^{(n_k)}_t$
which is zero for $r>t$.
\end{proof}

\section{Differentiability with respect to the initial data}
\lb{Diff}
Consider the equation
\aaa{
\lb{3}
Y_t = e^{tA} + \int_0^t e^{(t-s)A}\al'(X_s))Y_s\, ds + \int_0^t e^{(t-s)A}\sg'(X_s)\,Y_s\, dW_s.
}
which is obtained by formal differentiation of \rf{mild_SDE} with respect to the initial data,
and is written with respect to the derivative operator $Y_s$.
We would like to prove the existence of a solution to \rf{3} in the space of bounded
operators. We need the assumptions below.
\begin{enumerate}
\item[\textbf{A4}] 
$\al'(x)$ is bounded in $\mc L(E)$ and
$\gm(H,E)$, and $\sg'(x)$ is bounded in 
$\gm(H,\gm(H,E))$ and $\gm(H,\mc L(E))$.
\item[\textbf{A5}]
The restriction of the semigroup $e^{tA}$ to $H$ is a semigroup on $H$.
\end{enumerate}
For simplicity, we will use the same notations, i.e. $\al'(x)$, $\sg'(x)$, $e^{tA}$, 
for the restrictions to $H$.
\begin{thm}
 \label{th1}
Suppose A3, A4, and A5 are fulfilled. 
Then the solution $X_t(x)$ to \eqref{mild_SDE}
is Fr\'echet differentiable along $H$ with respect to the initial data $x$. 
The derivative operator $Y_t$ takes the form
$Y_t = e^{tA} + V_t$ where $V_t$ takes values in 
$\gm(H,E)$. Moreover, $Y_t$ is a solution to \rf{3},
and possesses a continuous path modification.
\end{thm}
\begin{proof}
First, consider the equation in $E$:
\aaa{
\lb{fre-d}
\xi_t = e^{tA}y + \int_0^t e^{(t-s)A}\al'(X_s)\xi_s\, ds + \int_0^t e^{(t-s)A}\sg'(X_s)\,\xi_s\, dW_s.
}
Since the derivatives $\al'(x)$ and $\sg'(x)$ are bounded
uniformly in $x\in E$, the proof of the existence of the solution
and its continuous path modification is the same 
as in Theorem \ref{thm0}.

Consider the operator $V_t = Y_t - e^{tA}$, and rewrite \rf{3} with respect to
$V_t$:
\mmm{
\lb{add4}
V_t = \int_0^t e^{(t-s)A}\al'(X_s))V_s\, ds + 
\int_0^t e^{(t-s)A}\sg'(X_s)\,V_s\, dW_s\\
+ \int_0^t e^{(t-s)A}\al'(X_s)\,e^{sA}\, ds
+ \int_0^t e^{(t-s)A}\sg'(X_s)\, e^{sA}\, dW_s.
}
Note that by Assumptions 4 and 5 and by the ideal property of $\gm(H,E)$,
$e^{(t-s)A}\al'(X_s)\, e^{sA}$ takes values in $\gm(H,E)$, and
$e^{(t-s)A}\sg'(X_s)\, e^{sA}$ takes values in $\gm(H,\gm(H,E))$. 
This, in particular, follows from the fact that $e^{(t-s)A}$, as a bounded operator $E\to E$, 
can be also regarded as
a bounded operator $\gm(H,E)\to\gm(H,E)$
whose norm is not bigger than $\|e^{(t-s)A}\|_{\mc L(E)}$. 
Therefore, as in the proof
of Theorem \ref{thhh2}, the stochastic integral $\int_0^t e^{(t-s)A}\sg'(X_s)\, e^{sA}\, dW_s$
takes values in $\gm(H,E)$, and
\aa{
\E\Big\|\int_0^t e^{(t-s)A}\sg'(X_s)\,e^{sA}\, dW_s\Big\|_{\gm(H,E)}^2
\lt C\, \E \int_0^t \|e^{(t-s)A}\sg'(X_s)\,e^{sA}\|^2_{\gm(H,\gm(H,E))}\, ds,
}
where $C>0$ is a constant.
Hence, the last two terms in \rf{add4} are bounded in $L_2(\Om,\gm(H,E))$.
The existence of a solution to \rf{add4} can be proved 
in the space $S_2(\gm(H,E))$ defined in Section \ref{existence}, 
in exactly the same way as we proved the existence of the solution to 
\rf{mild_SDE}. Moreover, the solution $V_t$ to \rf{add4} is unique 
and possesses a continuous path modification.
Equation \rf{add4} also implies that the process
$Y_t = e^{tA} + V_t$ solves \rf{3}. Indeed, 
both terms
$e^{(t-s)A}\sg'(X_s)\,V_s$ and $e^{(t-s)A}\sg'(X_s)\, e^{sA}$
take values in $\gm(H,\gm(H,E))$, and, therefore,
the stochastic integral   
$\int_0^t e^{(t-s)A}\sg'(X_s)\,Y_s\, dW_s$ is well-defined.
Hence, $Y_t$ verifies \rf{3}.

Now take a $y\in H$, and apply the both parts of \rf{3} to $y$.
We obtain that $Y_ty$ verifies \rf{fre-d}. But the solution to
\rf{fre-d} is unique in $S_2(E)$. From this and from
the continuity of paths it follows that for all $t\in [0,T]$
$Y_t y = \xi_t$ a.s. Therefore, $Y_t$ is the Fr\'echet derivative
of $X_t(x)$ with respect to $x$ along the space $H$.
\end{proof}
\section{The right inverse operator}
\lb{inverse}
In this section, under some additional assumptions, 
we prove the existence of the right inverse operator to $Y_t$.
We will need Lemma \ref{trace} below, proved in \cite{Ito_UMD}.
\begin{lem}
\label{trace}
Let $E$, $F$, $G$ be Banach spaces, and let
$\{e_n\}$ be an orthonormal basis of $H$.
Let $R\in \gm (H,E)$, $S\in \gm(H,F)$, and
$T \in  \mc L(E, \mc L(F,G))$. Then the sum
\aa{
\Tr_{R,S} T = \sum_{n=1}^\infty(TRe_n)(Se_n)
}
converges in $G$, does not depend on the choice
of the orthonormal basis, and
\aa{
\|\Tr_{R,S}T\|_G \lt \|T\|_{\mc L(E, \mc L(F,G))}
\|R\|_{\gm(H,E)}\|S\|_{\gm(H,F)}.
}
\end{lem}
Note that by this lemma, for each $x\in E$,
the sum $\Sg(x) = \sum_{i=1}^\infty \sg_i'(x)\sg_i'(x)$
converges in $\mc L(E,H)$  provided
that $\sg'(x)\in\gm(H,\mc L(E,H))$. Indeed,
\aa{
\|\Sg(x)\|_{\mc L(E,H)} 
= \big\|\sum_{i=1}^\infty \sg'(x)e_i\sg'(x)e_i\big\|_{\mc L(E,H)}
\lt \|\sg'(x)\|^2_{\gm(H, \mc L(E,H))}.
}
We will make additional assumptions:
\bi
\item[\textbf{A6}]
$e^{TA}:E\to E$ is an injectve map.
\item[\textbf{A7}]
There exists a Hilbert space $\td H$ containing $E$ as a subspace,
so that the canonical embedding of $E$ into $\td H$ is continuous, and
for all $x\in E$, $\sg'_i(x)$, $i=1,2,\ldots$, can be extended to $\td H$.
Moreover, each $\sg'_i(x)$ 
maps $e^{tA}\td H$ to $e^{tA}H$ for all $t\in [0,T]$, and 
for some constant $C>0$,
\aa{
\|e^{-TA} \sg'(x) e^{TA}\|_{\mc L_2(H,\mc L_2(\td H, H))}<C,
}
where $\mc L_2(H_1,H_2)$ denotes the space of the Hilbert-Schmidt operators
from a Hilbert space $H_1$ to another Hilbert space $H_2$.
\item[\textbf{A8}] For all $x\in E$, $\al'(x)$ 
maps $e^{tA}E$ to $e^{tA}H$ for all $t\in [0,T]$,
so that 
\aa{
\|e^{-TA} \big(\Sg(x)-\al'(x)\big) e^{TA}\|_{\mc L(E,H)} < C.
}
\ei
Since, by A6, all the operators $e^{tA}$ are injective maps $E\to E$, as well as $H\to H$ due to A5, 
one can speak
about the inverse operator $e^{-tA}$, in general unbounded, on $e^{tA}E$. 
Consider the equation
\eee{
\label{mild-adjoint}
Z_te^{tA} = I + \int_0^t Z_s \bigl(
\Sg(X_s)-\al'(X_s)\bigr)e^{sA}\, ds 
-\int_0^t Z_s\sg'(X_s) e^{sA}\, dW_s
}
which is obtained by a formal derivation of an SDE for $Y_t^{-1}$
and multiplying the both parts by $e^{tA}$ from the
right. Introducing the operator $R_t = Z_t e^{tA}$, we obtain the SDE for $R_t$:
\aaa{
\label{mild-adjoint-1}
R_t = I + \int_0^tR_s e^{-sA} \bigl(
\Sg(X_s)-\al'(X_s)\bigr)e^{sA}\, ds 
-\int_0^tR_s e^{-sA}\sg'(X_s) e^{sA}\, dW_s.
}
\begin{thm}
\label{th-h1}
Let Assumptions A5--A8 be fulfilled. 
Then, equation \eqref{mild-adjoint-1}
has a unique solution of the form $R_t = I + U_t$ where $U_t$ is $\mc L(E,H)$-valued.
Moreover, the operator $Z_t = R_te^{-tA}$, defined on $e^{tA}E$, is the 
right inverse to $Y_t$.
\end{thm} 
\begin{proof}
Written with respect to $U_t = R_t - I$, \rf{mild-adjoint-1} takes the form:
\mmm{
\label{mild-adjoint-2}
U_t = \int_0^tU_s e^{-sA} \bigl(\Sg(X_s)-\al'(X_s)\bigr)e^{sA}\, ds 
-\int_0^tU_s e^{-sA}\sg'(X_s) e^{sA}\, dW_s \\
+ \int_0^t e^{-sA} \bigl(\Sg(X_s)-\al'(X_s)\bigr)e^{sA}\, ds
-\int_0^t e^{-sA}\sg'(X_s) e^{sA}\, dW_s.
}
Note that, for a Hilbert-Schmidt operator $B: \td H\to H$ we have the following relation
between its different norms:
\aaa{
\lb{norms}
\|A\|_{\mc L(H)} \lt \|A\|_{\mc L(E,H)} \lt \|A\|_{\mc L_2(\td H,H)}. 
}
This allows us to solve \rf{mild-adjoint-2} in $\mc L(E,H)$.
Indeed, due to \rf{norms}, the stochastic integral
$\int_0^t e^{-sA}\sg'(X_s) e^{sA}\, dW_s$ is well-defined in $\mc L_2(\td H,H)$,
and, therefore, in $\mc L(E,H)$.
Define the map $\Gm: S_2(\mc L(E,H))\to S_2(\mc L(E,H))$, $U\mto \Gm(U)$, where $\Gm(U)$ 
equals to the right-hand side of \rf{mild-adjoint-2}.
For the stochastic integral in the first line of \rf{mild-adjoint-2}, we obtain:
\mm{
\E\Big\|\int_0^tU_s e^{-sA}\sg'(X_s) e^{sA}\, dW_s\Big\|^2_{\mc L(E,H)}
\lt 
\E\Big\|\int_0^tU_s e^{-sA}\sg'(X_s) e^{sA}\, dW_s\Big\|^2_{\mc L_2(\td H,H)}\\
\lt 
\E\int_0^t \|U_s\|^2_{\mc L(E,H)}\|e^{-sA}\sg'(X_s) e^{sA}\|^2_{\mc L_2(H,\mc L_2(\td H,H))}\, ds.
}
Therefore,
the stochastic integral $\int_0^tU_s e^{-sA}\sg'(X_s) e^{sA}\, dW_s$
takes values in $\mc L(E,H)$. Moreover, the map $\Gm$
has a fixed point in $S_2(\mc L(E,H))$ which can be proved
in exactly the same way as in Theorem \ref{th1}, and hence,
\rf{mild-adjoint-2} has an $\mc L(E,H)$-valued solution.
The solution $U_t$ to \rf{mild-adjoint-2} is also unique 
and possesses a continuous path modification.
It is easy to verify that $R_t = I + U_t$ solves \rf{mild-adjoint-1},
and, moreover, it is a unique solution.

Let us consider the equation:
\aaa{
\lb{34}
P_t = I + \int_0^t e^{-sA}\al'(X_s)e^{sA}P_s\, ds 
+ \int_0^t e^{-sA}\sg'(X_s)e^{sA}\,P_s\, dW_s.
}
Similar to \rf{3}, we can prove that \rf{34} has a solution
of the form $P_s = I + \td V_s$, where $\td V_s\in S_2(\gm(H,E))$,
and that $\td V_s$ is the unique solution to
\mmm{
\lb{35}
\td V_t = \int_0^t e^{-sA}\al'(X_s)e^{sA}\td V_s\, ds 
+ \int_0^t e^{-sA}\sg'(X_s)e^{sA}\,\td V_s\, dW_s\\
+\int_0^t e^{-sA}\al'(X_s)e^{sA}\, ds 
+ \int_0^t e^{-sA}\sg'(X_s)e^{sA}\, dW_s.
} 
Multiplying \rf{34} and \rf{35} by $e^{tA}$ from the left,
we obtain \rf{3} and \rf{add4}, respectively. By the
uniqueness
of the solution to \rf{add4} in $S_2(\gm(H,E))$, 
$V_t = e^{tA} \td V_t$ and, therefore, $Y_t = e^{tA} P_t$. 

Let us show that $P_tR_t=I$ on $H$.
To compute $P_tR_t$, we apply It\^o's formula to $\<R_ty,P_t^*y'\>$,
where $y\in H$, $y'\in E^*$:
\mm{
\<P_tR_ty,y'\> =  \<y,y'\> +
\int_0^t \<e^{-sA}\al'(X_s)e^{sA}P_sR_sy,y'\> ds \\
+ \int_0^t \<e^{-sA}\sg'(X_s)e^{sA}\,P_sR_sy,y'\> dW_s
-\int_0^t\<P_sR_s e^{-sA}\sg'(X_s) e^{sA}y,y'\> dW_s\\
+\int_0^t\<P_sR_s e^{-sA} \bigl(\Sg(X_s)-\al'(X_s)\bigr)e^{sA}y,y'\> ds \\
- \int_0^t \sum_{k=1}^\infty \<e^{-sA} \sg'_k(X_s)e^{sA}  P_sR_s 
e^{-sA} \sg'_k(X_s) e^{sA}y,y'\> ds.
}
Note that if we substitute $P_tR_t = I$, the above equation will be
satisfied. Denoting $P_tR_t - I$ by $Q_t$ we obtain the following SDE:
\mmm{
\label{Qsy}
Q_t = 
 \int_0^t e^{-sA}\al'(X_s)e^{sA}Q_s ds
+ \int_0^t Q_s e^{-sA}(\Sg(X_s) - \al'(X_s)) e^{sA} \, ds\\
+ \int_0^t e^{-sA} \sg'(X_s)e^{sA} Q_s  \, dW_s
- \int_0^t Q_s e^{-sA}\sg'(X_s) e^{sA}\, dW_s\\
- \int_0^t \sum_{k=1}^\infty e^{-sA} \sg'_k(X_s)e^{sA}  Q_s 
e^{-sA} \sg'_k(X_s) e^{sA}\, ds.
}
By the assumptions imposed on $\al'$ and $\sg'$, the right-hand side
takes values in $\gm(H,E)$. Therefore, \rf{Qsy} is a well-defined SDE 
in $\gm(H,E)$.
The usual stochastic integral estimates and Gronwall's inequality imply that
\aa{
\E \|Q_t\|_{\gm(H,E)}^2=0
}
for all $t\in [0,T]$. 
Hence, $Q_t = 0$ a.s.
This implies that $P_tR_t = I$ on $H$. 
By continuity
of paths, the set $\td \Om\sub\Om$, $\PP(\td \Om) = 1$, 
where $P_tR_t = I$, can be choosen the same for all $t\in [0,T]$. 
Note that on $H$, $P_t = I - P_tU_t$. But $I-P_tU_t$ takes
values in $\mc L(E)$ a.s. Therefore, $P_t$ can be a.s. extended 
to a bounded operator $E\to E$.
Thus, $P_tR_t = I$ everywhere on $\td \Om$. This implies that 
$Y_tZ_te^{tA}y = e^{tA}y$ a.s. for all $y\in E$.
Hence, $Y_tZ_t = I$ a.s. on $e^{tA}E$.
\end{proof}
\section{A version of H\"ormander's theorem}
\lb{Hormander}
For every $x\in D(A)$, where $D(A)$ denotes the domain of $A$,
we define $\sg_0(x) = Ax + \al(x) - \frac12\sum_{k=1}^\infty \sg'_k(x)\sg_k(x)$,
and note that the third summand is well-defined. Indeed,
application of Lemma \ref{trace}  
implies the convergence of the sum in $E$:
\aa{
\Big\|\sum_{k=1}^\infty \sg'(x)e_k\sg(x)e_k\Big\|_E\lt 
\|\sg'(x)\|_{\gm(H,\mc L(E))}\|\sg(x)\|_{\gm(H,E)}.
}
SDE \eqref{SDE} takes the form
\aa{
dX_t = \sg_0(X_t)dt + \sum_{k=1}^\infty \sg_k(X_t)\circ dW^k_t.
}
For two differentiable vector fields $V_1, V_2:E\to E$ the
Lie bracket $[V_1,V_2]$ is defined as in \cite{Nualart}.
If a vector field of the form $A^kx$, $k=1,2,\ldots,$ is involved in a Lie bracket, 
then $A$ is formally treated as a bounded operator
when we compute derivatives.
For example, if  $V: E\to D(A)$ is a vector field which is
Fr\'echet differentiable $E\to E$,
then, the Lie bracket
$[Ax,V(x)]:$\  $D(A) \to E$ is computed by the formula
\aa{
[Ax,V(x)] = AV(x) - V'(x)Ax.
}
For our version of H\"ormander's theorem,
we need Assumptions A9, A10, and H below:
\begin{enumerate}
\item[\textbf{A9}]
$\al,\sg_i$, $i=1,2,\ldots$, are 
infinitely differentiable functions
$E\to D(A^\infty)$, where $D(A^\infty) = \cap_{i=1}^\infty D(A^i)$;
the function $\sg': E\to \mc \gm(H,\mc L(E))$ is differentiable. 
\item[\textbf{A10}] $\sg$ is a map $E\to \mc L(H,e^{TA}E)$,
where  the Banach space $e^{TA} E$ is equipped with the norm $\|x\|_{e^{TA}E} = \|e^{-TA} x\|_E$.
\item[\textbf H](H\"ormander's condition)
The vector space spanned by the vector fields
\aa{
\sg_1, \sg_2, \ldots, \, [\sg_i,\sg_j], [\sg_i,[\sg_j,\sg_k]], \,  i,j,k = 0,1, \ldots
}
evaluated at point $x\in D(A^\infty)$, is dense in $E$.
\end{enumerate}
Note that under Assumption A9, all the Lie brackets 
in Assumption H are well-defined as vector fields 
$D(A^\infty)\to D(A^\infty)$. In the following, 
we will need Lemma \ref{lem4} below.
\begin{lem}
\label{lem4}
Let $V: E\to E$ be a $C^2$-vector field.
Under Assumption A9, the term
$\bigl\{ [\sg_0,V] + \frac12\sum_{k=1}^\infty [\sg_k,[\sg_k,V]]\bigr\}(x)$,
$x\in E$, is well-defined. Moreover,
\aa{
-\frac12\, \Bigl[\,\sum_{k=1}^\infty \sg'_k(x)\sg_k(x),V\Bigr](x) + 
\frac12\sum_{k=1}^\infty [\sg_k,[\sg_k,V]](x)\\ =\sum_{k=1}^\infty
\bigl(-\sg'_k(V'\sg_k) + \frac12 \, V''\sg_k\sg_k + \sg'_k(\sg'_kV)
\bigr).
}
\end{lem}
\begin{proof}
Let us compute the sum of two Lie brackets for a fixed $k$.
\mmm{
\lb{last-line}
-\frac12\,[\sg'_k(x)\sg_k(x),V] + \frac12\,[\sg_k,[\sg_k,V]] = 
\frac12\bigl(\sg''_k\sg_k V + \sg'_k(\sg'_k V) - V'(\sg'_k\sg_k)\\
+\sg'_k(\sg'_kV) - \sg'_k(V'\sg_k) - \sg''_k V\sg_k - \sg'_k(V'\sg_k)
+ V''\sg_k\sg_k + V'(\sg'_k\sg_k)\bigr)\\
= -\sg'_k(V'\sg_k) + \frac12 \, V''\sg_k\sg_k + \sg'_k(\sg'_kV).
}
By Lemma \ref{trace}, for the first term in the last line of \rf{last-line}
we have the estimate:
\aa{
\Big\|\sum_{k=1}^\infty \sg'(x)e_k(V'\sg)(x) e_k)\Big\|_E
\lt 
\|\sg'(x)\|_{\gm(H,\mc L(E))} \|V'(x)\sg(x)\|_{\gm(H,E)}.
}
For the third term in the last line of \rf{last-line},
we obtain:
\aa{
\Big\|\sum_{k=1}^\infty \sg'(x)e_k(\sg' V)(x)e_k\Big\|_E
\lt \|\sg'(x)\|_{\gm(H,\mc L(E))} \|(\sg'V)(x)\|_{\gm(H,E)}.
}
Finally, the estimate of the second term in the third line of \rf{last-line} is:
\aa{
\Big\|\sum_{k=1}^\infty [(V''\sg)(x) e_k\sg(x) e_k]\Big\|_E
\lt \|V''(x)\|_{\mc L(E,\mc L(E))}\, \|\sg(x)\|^2_{\gm(H,E)}.
}
These estimates imply that we can take summations in $k$ of the both parts in
\rf{last-line}. The additional two estimates
\aa{
\sum_{k=1}^\infty \Big\|\sg''(x)V(x)e_k \sg(x)e_k\Big\|_E \lt
\|\sg''(x)V(x)\|_{\gm(H,\mc L(E))}
\|\sg\|_{\gm(H,E)}
}
and
\aa{
\Big\|\sum_{k=1}^\infty [(V'\sg')(x) e_k\sg(x) e_k]\Big\|_E
\lt  \|(V'\sg')(x)\|_{\mc \gm(H,\mc L(E))}\, \|\sg(x)\|_{\gm(H,E)}
}
imply that 
the term $[\sg_0,V]$ is well-defined. This, in turn, implies that the term
$\frac12\sum_{k=1}^\infty [\sg_k,[\sg_k,V]](x)$ is well-defined as well.
\end{proof}
\begin{lem}
Let $X_t$ be a mild solution to SDE \eqref{SDE}
with the initial condition  $x\in D(A)$, and let A9
be fulfilled. Then $X_t$ is also a strong solution to \eqref{SDE}.
\end{lem}
\begin{proof}
The statement of the lemma can be verified by taking stochastic differentials
of the both parts of \rf{SDE}.
\end{proof}
\begin{lem}
\lb{lee-2}
Let Assumptions A1, A2, and A10 be fulfilled. Further
let $x\in D(A)$, and $V:E\to e^{TA}E$ be a vector field
which is a $C^2$-smooth function $E\to E$. 
Then, if $X_t$ is a strong solution to \eqref{SDE} and 
$Z_t$ is the solution to \eqref{mild-adjoint}, it holds that
\mmm{
\label{form_lem2}
Z_tV(X_t) = V(x) + \sum_{k=1}^\infty \int_0^t Z_s[\sg_k,V](X_s)\,dW^k_s \\
 +\int_0^t Z_s\big(
[\sg_0,V] + \frac12\sum_{k=1}^\infty [\sg_k,[\sg_k,V]]
\big)(X_s)\, ds.
}
\end{lem}
\begin{proof}
By It\^o's formula (see \cite{Belopolskaya}, \cite{brz}),
\mm{
V(X_t)  = V(x)  + 
\int_0^t V'(X_s)(AX_s + \al(X_s)) ds + \int_0^t V'(X_s) \sg(X_s)\, dW_s  \\
+ \frac12 \int_0^t \sum_{k=1}^\infty V''(X_s) \sg_k(X_s) \sg_k(X_s)\, ds.
}
Equation \rf{mild-adjoint} implies that for $x\in e^{TA}D(A)$,
\aa{
Z_t x = x - \int_0^t Z_sAx\, ds + \int_0^t Z_s(\Sg(X_s) - \al'(X_s))x\, ds
- \int_0^t Z_s\sg'(X_s)x\, dW_s.
}
Applying It\^o's formula to $Z_tV(X_t)$ we obtain: 
\mm{
Z_tV(X_t) = V(x) 
+\int_0^t Z_s\big(\Sg(X_s)-A - \al'(X_s)\big)V(X_s)\, ds\\
+ \int_0^t Z_s\Big(V'(X_s)(AX_s + \al(X_s)) + \frac12 \sum_{k=1}^\infty
V''(X_s)\sg_k(X_s)\sg_k(X_s)\Big)\, ds\\
-\int_0^t \sg'(X_s)Z_s V(X_s)\, dW_s
+ \int_0^t Z_sV'(X_s)\sg(X_s)\, dW_s\\
-\int_0^t \sum_{k=1}^\infty Z_s\sg'_k(X_s)V'(X_s)\sg_k(X_s)\, ds
= V(x) +\\
\int_0^t Z_s[AX_s + \al(X_s),V(X_s)]\, ds 
+  \sum_{k=1}^\infty \int_0^t Z_s[\sg_k,V](X_s)\, dW_s \\ 
+ \int_0^t \sum_{k=1}^\infty Z_s
[-\sg'_k(V'\sg_k) + \frac12 \, V''\sg_k\sg_k + \sg'_k(\sg'_kV)](X_s)\,ds.
}
By Lemma \ref{lem4}, the right-hand side of the 
above relation equals the right-hand side of \eqref{form_lem2}.
\end{proof}
The main result of this paper is the following version
of H\"ormander's theorem.
\begin{thm}
\label{hth}
Let Assumptions A1-A10 and H\"ormander's condition H be satisfied.
Further let $F : E \to \Rnu^k$ be a
bounded linear operator of rank $k$.
Then, for any fixed $t$, the probability distribution of 
$FX_t$ is absolutely continuous with respect to
the Lebesgue measure on $\Rnu^k$.
\end{thm}
\begin{proof}
In Theorem \ref{thhh2} we proved that the Malliavin derivative
$D_rX_t$, where $r\lt t$ are fixed, verifies equation \rf{SDE-MD}.
The uniqueness of the solution to \rf{SDE-MD} in $S_2(\gm(H,E))$
follows from the results of Section \ref{Diff}.
Let us note that the process $Y_t Z_r \sg(X_r)$ also satisfies \rf{SDE-MD}.
Indeed, \rf{3} implies:
\aa{
Y_t = e^{(t-r)}Y_r + \int_r^t e^{(t-s)A} \al'(X_s)Y_s\, ds
+\int_0^t e^{(t-r)A}\sg'(X_s) Y_s\, dW_s. 
}
Noticing that $Y_tZ_r\sg(X_r)$ takes values in $\gm(H,E)$, we multiply
the both sides of the above equation by
$Z_r\sg(X_r)$. Taking into consideration that 
$Y_rZ_r\sg(X_r)=\sg(X_r)$, we obtain that 
$Y_t Z_r \sg(X_r)$ verifies \rf{SDE-MD}.
By uniqueness of the solution to \rf{SDE-MD} in $S_2(\gm(H,E))$,
$D_rX_t = Y_tZ_r\sg(X_r)$.

Fix a $t>0$.
Let $X_s$ be the solution of \eqref{SDE}, and let $\xi_s = F X_s$.
By Lemma \ref{chain-lem}, the Malliavin derivative $D_r \xi_t$, $r<t$,
equals to
\aaa{
 \label{Malliavin_deriv}
 D_r \xi_t = F \, Y_t \, Z_r  \,\sg(X_r).
}
Further let $\gm_t$ denote the Malliavin covariance matrix of 
$\xi_t$. Using relation \eqref{Malliavin_deriv} we can write down $\gm_t$ in the form:
\aa{
\gm_t = (F\circ Y_t)C_t(F\circ Y_t)^*
}
where the operator $C_t : E^* \to E$ is defined as
\aa{
 C_t = \int_0^t Z_r \sg(X_r) \sg(X_r)^* Z_r^* \, dr.
 }
By Theorem 2.1.2. of \cite{Nualart}, the statement of the theorem
will be implied by the invertibility of $\gm_t$.
In order to show that $(\gm_t x, x)_{\Rnu^k}>0$ for all $x\in \Rnu^k$,
it suffices to prove that for all $\ffi \in E^*$, $\ffi \ne 0$,
\aaa{
\label{C_t_pos} 
\< C_t \ffi, \ffi \> > 0
}
with probability one.  
Indeed,, for every $x\in \Rnu^k$,
\aa{
(\gm_t x, x)_{\Rnu^k} = \< C_t Y_t^* F^*x, Y_t^* F^*x\>.
}
Note that ${\rm ker}\, F^* = \{0\}$. Indeed, assume that there is a
$z\in \Rnu^k$, $z\ne 0$, such that $F^* z = 0$. Then, for all $y\in E$,
$(Fy, z)_{\Rnu^k} = \<y, F^*z\> = 0$. 
Thus, $z$ is orthogonal to $\im F$ in $\Rnu^k$ which contradicts to 
the assumption that $F$ has rank $k$.
Also, note that  ${\rm ker}\, Y_t^* = \{0\}$. Indeed, 
$Y_t^*y'=0$ implies $Z_t^*Y_t^*y' = y' = 0$. 
Hence, $Y_t^* F^*x \in E^*$ is non-zero if $x\in\Rnu^k$ is non-zero.

Thus we have to prove \eqref{C_t_pos}. 
Let us assume that there exists a $\ffi_0 \ne 0$ such that
\eee{
\label{hyp}
P\{\< C_t \ffi_0, \ffi_0 \> =0\} > 0.
}
Take a $\ffi\in E^*$. We have:
\mmm{
 \label{88}
 \< C_t\ffi, \ffi \> =  \int_0^t \|\sg(X_r)^* Z_r^* \ffi\|_H^2 \,dr =
 \int_0^t \sum_{k=1}^\infty  (e_k, \sg(X_r)^* Z_r^* \ffi)_H^2 \, dr \\
 =\int_0^t \sum_{k=1}^\infty  \< Z_r\,\sg_k(X_r), \ffi \>^2 \, dr.
}
Define random spaces $K_s\sub E$:
\[
K_s = \ovl{\spn\{ Z_\zeta \sg_k(X_\zeta); \, \zeta \in [0,s], k\in \Nnu\}}.
\]
The family of vector spaces $\{K_s, s \gt 0\}$ is increasing. 
Let $K_{0+} = \cap_{s>0} K_s$.
By the Blumental zero-one law, $K_{0+}$ is deterministic with probability one,
since every random variable $b\in K_{0+}$ is constant with probability one.
Let $N>0$ be an integer, and let $N_s$ be the codimension of $K_{0+}$ in
$K_s$, possibly infinite. Consider the non-decreasing adapted 
process $\{\min\{N,N_s\}, s>0\}$,
and the stopping time
\[
S = \inf\bigl\{ s>0: \, \min\{N, N_s\}>0 \bigr\}.
\]
Note that $\PP\{S>0\}=1$. 
Indeed, if we assume that $\PP\{S=0\} > 0$, it would imply that 
with a positive probability the codimension of $K_{0+}$ in $K_s$ 
is positive for any $s>0$.
The latter fact  implies that with a positive probability 
the codimension of $K_{0+}$ in 
$\cap_{s>0}K_s$ is positive as well, which is a contradiction. 

Next, note that $K_{0+} \ne E$. Indeed, if $K_{0+} = E$,
then $K_s = E$ for all $s>0$.
Therefore, if $\ffi\in E^*$ is such that  $\< C_t\ffi, \ffi \> =0$,
then $\< Z_r\sg_k(X_r), \ffi\> = 0$ for all $r\in [0,t]$ and for all $k\in \Nnu$
by \rf{88}.
This implies that $\ffi$ is zero on $K_s$, and hence, $\ffi = 0$,
which contradicts to hypothesis   \eqref{hyp}.

Take a non-zero functional $\ffi\in E^*$ containing $K_{0+}$ in its kernel.
Note that for all $s < S$, $\ffi(K_s)=0$, and hence,
\aaa{
\label{77}
\<Z_s\sg_k(X_s),\ffi\> = 0 \quad \text{for all} \;  k \; \text{and} \; s<S.
}
Introduce the following sets of vector fields:
\aa{
 \Sg_0 = & \{\sg_1, \sg_2, \ldots, \sg_k, \ldots \}\\
 \Sg_n = & \{[\sg_0,V], \, [\sg_k,V], k\in \Nnu, \, V\in \Sg_{n-1} \}\\
 \Sg = & \cup_{n=1}^\infty \Sg_n
}
and
\aa{
 \Sg'_0 = & \Sg_0, \\
 \Sg'_n = & \{ [\sg_k,V], k=1,2, \ldots, \, V\in \Sg'_{n-1}; \\
            & [\sg_0,V] + \frac12 \sum_{k=1}^\infty [\sg_k,[\sg_k,V]], \, V\in \Sg'_{n-1}\}\\
\Sg' = & \cup_{i=1}^\infty \Sg'_n.
}
Let $\Sg_n(x)$ and $\Sg'_n(x)$ denote the subspaces of $E$ obtained from 
$\Sg_n$ and $\Sg'_n$, respectively, by 
evaluating the vector fields of the latters at point $x\in D(A^\infty)$. 
Note that the vector fields from $\Sg_n$ (resp. $\Sg'_n$) 
are well-defined on $D(A^{n-1})$.
Clearly,
the spaces spanned on $\Sg(x)$ and $\Sg'(x)$ coincide with each other.
They also coincide with the space $E$ by Assumption H.
Let us show that 
\aaa{
\label{99}
\ffi(\Sg'_n(x))=0 \quad  \text{for all} \; n. 
}
By Assumption H,
this will imply that $\ffi = 0$, and hence, it will be a contradiction.
Property \eqref{99} is implied by the following stronger property:
\aaa{
\label{H_contradiction}
\<Z_s V(X_s),\ffi\> = 0 \quad \text{for all} \; s<S, \, V\in \Sg'_n, \, n\gt 0.
}
We show \eqref{H_contradiction} by induction on $n$. For $n=0$, \eqref{H_contradiction}
follows from \eqref{77}. We assume that \eqref{H_contradiction} holds for $n-1$, and show
that it holds for $n$. Let $V\in \Sg'_{n-1}$. 
Note that if $V$ takes values in $e^{tA}E$,
then $[\sg_k,V]$ and $[\sg_0,V] + \frac12 \sum_{k=1}^\infty [\sg_k,[\sg_k,V]]$
also take values in $e^{tA}E$. By Lemma \ref{lee-2},
\mm{
0 = \<Z_s V(X_s),\ffi\> = 
\<V(x),\ffi\> + \sum_{k=1}^\infty \int_0^s \< Z_r [\sg_k,V](X_r),\ffi\> dW^k_r \\
 +\int_0^s \< Z_r
 \Bigl\{ 
 [\sg_0,V] + \frac12\sum_{k=1}^\infty [\sg_k,[\sg_k,V]]
\Bigr\}(X_r),\ffi\> dr
}
which holds for all $s<S$. Since $\<V(x),\ffi\>=0$, it implies that
for all $s < S$, the quadratic variation of the martingale part
and the bounded variation part of this semimartingale must be zero.
This proves \eqref{H_contradiction}.
\end{proof}

{\bf Acknowledgements.} This research was funded by
the European Regional Development Fund through the program COMPETE
and by the Portuguese Government through the FCT (Funda\c{c}\~ao
para a Ci\^encia e a Tecnologia) under the project
PEst-C/MAT/UI0144/2011.

\end{document}